\documentclass[12pt,letterpaper]{article}
\usepackage[hebrew,english]{babel}

\usepackage{amsfonts,amssymb}
\usepackage{amsmath}
\usepackage{acronym}

\numberwithin{equation}{section}

\newtheorem{theorem}{Theorem}

\newtheorem{prop}{Proposition}
\newtheorem{definition}{Definition}
\newtheorem{remark}{Remark}
\newtheorem{corollary}{Corollary}

\newtheorem{lemma}{Lemma}

\newcommand{\qed}{\(\Box\)}
\newcommand{\N}{\mathbb{N}}
\newcommand{\Z}{\mathbb{Z}}
\newcommand{\openrm}{\mathrm{(}}
\newcommand{\closerm}{\mathrm{)}}
\newcommand{\dd}{\mathrm{d}}

\def\dim{\hbox{\rm dim}\,}

\numberwithin{lemma}{subsection} \numberwithin{theorem}{subsection}
\numberwithin{prop}{subsection} \numberwithin{remark}{subsection}
\numberwithin{definition}{subsection} \numberwithin{corollary}{subsection}

\title{Solvability of linear equations within weak mixing sets}
\author{Alexander Fish}
\begin{document}

\maketitle

\begin{abstract} We introduce a new class of ``random" subsets of natural
numbers,  WM sets. This class contains normal sets (sets whose characteristic function is a normal 
binary sequence). We establish necessary and sufficient conditions for
solvability of systems of linear equations within every WM set and within every normal set. We also show that any partition-regular system of linear equations with integer coefficients is solvable in any WM set.
\end{abstract}

\section{Introduction}

\subsection{Algebraic patterns within subsets of  \(\mathbb{N}\)}

\noindent  We use extensively  the notion of ``algebraic pattern".
 By an algebraic pattern we mean a solution of a diophantine system of equations. 
For example, an arithmetic
progression of length \( k \) is an algebraic pattern corresponding to the following diophantine system:
\[
2 x_i = x_{i-1}+ x_{i+1}, i=2,3,\ldots,k-1.
\]

\noindent We investigate the problem of finding linear algebraic patterns (these correspond to linear systems) within a family of  subsets of natural numbers satisfying some asymptotic conditions.

\noindent For instance,  by Szemer\'{e}di theorem, subsets of  positive upper Banach density 
(all \( S \subset \mathbb{N}: d^*(S) > 0 \), where \( d^*(S) = \limsup_{b_n - a_n \to \infty} \frac{|S \cap [a_n,b_n]|}{b_n-a_n+1}\)) contain the pattern of an arithmetic progression of any finite length (see \cite{szemeredi}). 

  


%

\noindent On the other hand, \textit{Schur patterns}, namely triples of the form  \( \{ x,y,x+y\}  \), which correspond to solutions of the so-called Schur equation, \( x+y = z \),  do not necessarily occur in sets of positive upper density. For example, the odd numbers do not contain this pattern.  But if we take a random subset of \( \mathbb{N} \) by picking natural numbers with  probability \( \frac{1}{2} \) independently, then  this  set  contains the Schur pattern with probability \( 1 \).

\noindent There is a deterministically defined analog of a random set - a normal set.  To define a normal set we recall the notions of a normal infinite binary sequence and of a normal number.

\noindent  \textit{An infinite \( \{0,1\}\)-valued sequence \( \lambda \) is called a
\textbf{normal sequence} if
every finite binary word \( w \) occurs in \( \lambda \) with  frequency \(\frac{1}{2^{|w|}}\), where \( |w| \)
is the length of \( w \).} 

\noindent
The more familiar notion is that of a normal number \( x \in [0,1]\).
 If to a number  \( x \in [0,1]\) we associate its dyadic expansion  \( x = \sum_{i=1}^{\infty} \frac{x_i}{2^i} \, \) with \(x_i \in \{0,1\}\), then \( x \) is called a normal number if
  the sequence \( (x_1,x_2,\ldots,x_n,\ldots) \) is a normal sequence.

\begin{definition}
\textit{A set \( S \subset \mathbb{N}\) is called \textbf{normal} if the 0-1 sequence \( 1_S\) \textnormal{(\(1_S(n) = 1 \Leftrightarrow n \in S \))} is normal. }
 \end{definition}


\noindent Normal sets exhibit a non-periodic, ``random" behavior.
 We notice that if \( S \) is a normal set then \( S - S \) contains \( \mathbb{N} \).
  Therefore, the equation
 \[ z-y=x \] 
 is solvable within every normal set. This implies that
 every normal set
 contains Schur patterns.

\noindent  Normal sets are related to a class of dynamical systems displaying maximal randomness; namely Bernoulli systems. 
In this work we investigate occurrence of linear patterns in sets corresponding to dynamical systems with a lower degree of randomness, so called \textit{weakly mixing} dynamical systems.   
The sets we obtain will be called WM sets. We will make this precise in the next section. 

\noindent In the present paper we treat the following problem:

\noindent
\textit{Give a complete characterization of the linear algebraic patterns which occur in all WM sets.}

\noindent 
\begin{remark}
It will follow from our definition of a WM set, that any normal set is a WM set.
\end{remark}

The problem of the solvability of a nonlinear equation or  system of equations is beyond the limits of the technique used in this paper.  Nevertheless, some particular equations might be analyzed. In \cite{fish1} it is shown that there exist normal sets in which the multiplicative Schur equation \( xy = z \) is not solvable.

\subsection{Generic points and WM sets}\label{intro} For a formal definition of  WM sets 
 we need the notions of
measure preserving systems and of  generic points.

\begin{definition}
Let \( X \) be a compact metric space, \( \mathbb{B} \)  the Borel
\( \sigma\)-algebra on \( X \); let \( T:X \rightarrow X \) be a
continuous map and \( \mu \) a probability measure on \(
\mathbb{B} \). The quadruple \( (X,\mathbb{B},\mu,T) \) is called a
\textbf{measure preserving system} if for every \( B \in
\mathbb{B} \) we have \( \mu(T^{-1}B) = \mu(B) \).
\end{definition}

\noindent
 For a compact metric space \( X \) we denote by \( C(X)
\) the space of continuous functions on \( X \) with the uniform
norm.
\begin{definition} Let \( (X,\mathbb {B},\mu ,T) \) be a measure
preserving system. A point \( \xi \in X \) is called
\textbf{generic} for the system \( (X,\mathbb {B},\mu ,T) \) if for any \( f\in C(X) \) we have
\begin{equation}
\label{limit} \lim _{N\rightarrow \infty }\frac{1}{N}\sum
^{N-1}_{n=0}f(T^{n}\xi )=\int _{X}f(x)d\mu (x).
\end{equation}
\end{definition}

\noindent
\textbf{Example:}
Consider the \textit{Bernoulli system}: \( (X = \{0,1\}^{\mathbb{N}_0}, \mathbb{B}, \mu, T) \), where \( X \)  is endowed with the Tychonoff topology, \( \mathbb{B} \) is Borel \( \sigma\)-algebra on \( X \), \( T \) is the shift to the left, \( \mu \) is the product measure of \( \mu_i\)'s where \( \mu_i(0)=\mu_i(1) = \frac{1}{2} \)
and \( \N_0 = \N \cup \{0\}\).
An alternative definition of a normal
set which is purely dynamical is the following.

 A set \( S \) is normal if and only if 
the sequence \( 1_S \in \{0,1\}^{\mathbb{N}_0} \) is a generic
point of the foregoing Bernoulli \( \left( \frac{1}{2}, \frac{1}{2} \right) \) system. 
\vspace{0.1in}

\noindent   The notion of a WM set generalizes that of a normal set, where the role
played by  Bernoulli dynamical system is taken over by dynamical systems  of more general
 character.
 \vspace{0.1in}
 
\noindent
Let \( \xi (n) \) be any \( \{0,1\}- \)valued sequence. There is a
natural dynamical system \( (X_{\xi },T) \) connected to the
sequence \( \xi \):

\noindent On the  compact space \( \Omega
=\{0,1\}^{\mathbb {N}_0} \)  endowed with the Tychonoff topology,
we define a continuous map \( T:\Omega \longrightarrow \Omega  \)
by  \(
(T\omega )_{n}=\omega _{n+1} \). Now for any \( \xi \) in \(
\Omega \) we define 
\[ X_{\xi } = \overline{\{T^{n}\xi \}_{n\in \mathbb {N}_0}} \subset \Omega.
\]

\noindent Let \( A \) be a subset of \( \mathbb {N} \). Choose \(
\xi =1_{A} \) and assume that for an appropriate measure \( \mu \),
the point \( \xi \) is generic for \( (X_{\xi },\mathbb {B},\mu
,T) \). We can attach to the set \( A \) dynamical properties
associated with the system \( (X_{\xi },\mathbb {B},\mu ,T) \).


\noindent
We recall the notions of ergodicity, total ergodicity and weak-mixing in ergodic theory:
\begin{definition}
A measure preserving system \( (X,\mathbb{B},\mu,T) \) is called
\textbf{ergodic} if every \( A \in \mathbb{B} \) which is
invariant under \(T\), i.e. \( T^{-1}(A) = A \), satisfies \(
\mu(A) = 0 \) or \(1 \).
\newline
A measure preserving system \( (X,\mathbb{B},\mu,T) \)  is called
\textbf{totally ergodic} if for every \( n \in \mathbb{N} \) the
system \( (X,\mathbb{B},\mu,T^n) \) is ergodic.
\newline
A measure preserving system \( (X,\mathbb{B},\mu,T) \) is called
\textbf{weakly mixing} if the system \( (X\times X,\mathbb{B}_{X
\times X},\mu \times \mu,T \times T) \) is ergodic.
\end{definition}

\noindent 
In our discussion of WM sets corresponding to weakly mixing systems, we shall add the proviso that the dynamical system in question not be the trivial \(1\)-point system supported on the point \( x \equiv 0 \). This implies that the ``density" of the set in question be positive.

\begin{definition}
Let \( S \subset \mathbb{N} \). If the limit of \( \frac{1}{N} \sum_{n=1}^{N} 1_S(n) \) exists
 as \( N \rightarrow \infty \) we call it the \textbf{density} of \( S \) and denote by \( d(S) \).
\end{definition}

\begin{definition}
A subset \( S \subset \mathbb{N} \) is called a \textbf{WM set} if
\( 1_S \) is a generic point of the weakly mixing
system \( (X_{1_S},\mathbb{B},\mu ,T) \) and \( d(S) > 0 \).
\end{definition}


\subsection{Solvability of linear diophantine systems within WM
sets and normal sets}\label{sub_sect_linear_dioph_syst}

 \noindent Our main result is  a complete
characterization of  linear systems of diophantine equations
which are solvable within every WM set. The characterization is given by describing affine subspaces of $\mathbb{Q}^k$ which intersect  $A^k$, for any WM set $A \subset \mathbb{N}$.  
\begin{theorem}
\label{main_thm_lin}
An affine subspace of $\mathbb{Q}^k$ intersects $A^k$ for every WM set $A \subset \mathbb{N}$
if and only if
it contains a set of the form 
\[
\{ n\vec{a} + m\vec{b} + \vec{f} \, | \, n,m \in \mathbb{N} \},
\]
where \( \vec{a}, \vec{b}, \vec{f} \) have the following description:
\newline
 \( \vec{a} = (a_1,a_2,\ldots,a_k)^t \),
\(\vec{b} = (b_1,b_2,\ldots,b_k)^t \in \mathbb{N}^k\), \( \vec{f} = (f_1,f_2,\ldots,f_k)^t \in \mathbb{Z}^k \) and there exists a partition 
  \( F_1,\ldots,F_l \) of \( \{1,2,\ldots,k\}\) such that:
\newline
a\(\closerm\) for every \( r \in \{1,\ldots,l\}\) there exist \( c_{1,r},c_{2,r} \in \mathbb{N} \), such that
for every \( i \in F_r \) we have \( a_i = c_{1,r} \, , \, b_i = c_{2,r} \)
and for every \(j \in \{1,\ldots,k\} \setminus F_r \) we have
\[
\det \left( \begin{array}{cc}
                    a_j & b_j \\
                    c_{1,r} & c_{2,r}  \\
            \end{array}
    \right) \neq 0.
\]
\newline
b\(\closerm\) 
\[ \forall r \in \{1,2,\ldots,l\} \, \exists c_r \in \Z \,\, such \,\, that \,\, \forall i \in F_r \, : \, f_i = c_r. \]
\end{theorem}

\noindent 
We also classify all affine subspaces of $\mathbb{Q}^k$ which intersect $A^k$ for any normal set $A \subset \mathbb{N}$.
\begin{theorem}
\label{normal_theorem}
An affine subspace of $\mathbb{Q}^k$ intersects $A^k$ for every normal set $A \subset \mathbb{N}$ if and only if
it contains a set of the form 
\[
\{ n\vec{a} + m\vec{b} + \vec{f} \, | \, n,m \in \mathbb{N} \},
\]
where \( \vec{a}, \vec{b}, \vec{f} \) have the following description:
\newline
 \( \vec{a} = (a_1,a_2,\ldots,a_k)^t \),
\(\vec{b} = (b_1,b_2,\ldots,b_k)^t \in \mathbb{N}^k\), \( \vec{f} = (f_1,f_2,\ldots,f_k)^t \in \mathbb{Z}^k \) and there exists a partition 
  \( F_1,\ldots,F_l \) of \( \{1,2,\ldots,k\}\) such that for every \( r \in \{1,\ldots,l\}\) there exist \( c_{1,r},c_{2,r} \in \mathbb{N} \), such that
for every \( i \in F_r \) we have \( a_i = c_{1,r} \, , \, b_i = c_{2,r} \)
and for every \(j \in \{1,\ldots,k\} \setminus F_r \) we have
\[
\det \left( \begin{array}{cc}
                    a_j & b_j \\
                    c_{1,r} & c_{2,r}  \\
            \end{array}
    \right) \neq 0.
\]
\end{theorem}

\noindent A family of linear algebraic patterns that has been studied previously are the ``partition regular" patterns. These are patterns which  for any finite partition of  \( \N \): \( \N = C_1 \cup C_2 \cup \ldots \cup C_r \), the pattern necessarily occurs in some \( C_j \). (For example by van der Waerden's theorem, arithmetic progressions are partition regular and by Schur's theorem the Schur pattern is also partition regular).  A theorem of Rado gives a complete characterization of such patterns. We will show 
 in Proposition \ref{Rado_analog}
 that every linear algebraic pattern which is
partition-regular  
occurs in every WM set. 

\noindent It is important to mention that if we weaken the requirement of weak mixing to total ergodicity, then in the resulting family of sets, Rado's patterns need not necessarily occur.  For example, for 
\( \alpha \not \in \mathbb{Q} \) 
the set
\[
S = \left\{ n \in \mathbb{N}| \,  n\alpha\,  \, (\!\!\!\!\!\mod 1) \in \left[\frac{1}{3},\frac{7}{12}\right] \right\}
\] 
is \textit{totally ergodic}, i.e., \( 1_S \) is a generic point for a totally ergodic system and the density of \( S \) is positive,  but the equation \(x+y=z\) is not solvable within \( S \). 

\noindent In the  separate paper \cite{Fish_sumset_phenomenon} we will address the question of solvability of more general algebraic patterns, not necessarily linear, in totally ergodic and WM sets.

\noindent The structure of the paper is the following. In Section \ref{lin_proof_suff} we prove the direction ``\( \Leftarrow \)" of Theorems \ref{main_thm_lin} and \ref{normal_theorem}. In Section  \ref{lin_useful_constructions}, by use of a probabilistic method, we prove the direction ``\( \Rightarrow\)" of Theorems \ref{main_thm_lin} and \ref{normal_theorem}.  
In Section  \ref{sub_sect_Rado} we show that every linear system which is solvable in one of the cells of any finite partition of \( \N \) is also solvable  within every WM set. The paper ends with Appendix in which we collected proofs of technical statements which have been used in Sections \ref{lin_proof_suff} and \ref{lin_useful_constructions}.
\subsection{Acknowledgments}

\noindent This paper is a part of the author's Ph.D. thesis. I thank my advisor Prof. Hillel Furstenberg for introducing me to ergodic theory and for many useful ideas which I learned from him. I thank Prof. Vitaly Bergelson for fruitful  discussions and valuable suggestions. Also, I would like to thank an anonymous referee for numerous valuable remarks.

\numberwithin{lemma}{section} \numberwithin{theorem}{section}
\numberwithin{prop}{section} \numberwithin{remark}{section}
\numberwithin{definition}{section} \numberwithin{corollary}{section}

\section{Proof of Sufficiency}\label{lin_proof_suff}

\noindent
\textbf{Notation:} \textit{We introduce the scalar product of two vectors \( v,w \) of 
length \( N \) as follows:
\[
<v,w>_N \doteq \frac{1}{N} \sum_{n=1}^N v(n)w(n).
\]
We denote by \( L^2(N) \) the (finite-dimensional) Hilbert space of all real vectors of
 length \( N \) with the aforementioned scalar product.
\newline
  We define: \(
\parallel{w}\parallel_N^2 \doteq <w,w>_N \).}

\noindent First we state the following proposition which will prove useful 
in the proof of the sufficiency of the conditions of  Theorem \ref{main_thm_lin}.
\begin{prop}
\label{main_prop}
 Let \( A_i \subset \mathbb{N} \) \( \rm( \)\( 1 \leq i \leq k \)\( \rm) \) be  WM sets. Let \newline
 \(\xi_i(n) \doteq 1_{A_i}(n) - \dd(A_i) \), where \( \dd(A_i) \) denotes
density of \( A_i \). Suppose there are  \( (a_1,b_1),(a_2,b_2),
\ldots, (a_k,b_k) \in \mathbb{Z}^2 \), such
that \( a_i
> 0, \, 1 \leq i \leq k \), and for every \( i \neq j \)
\begin{displaymath}
\det \left( \begin{array}{cc}
a_i & b_i \\
a_j & b_j  \\
\end{array}
\right) \neq 0.
\end{displaymath}
 Then for every \( \varepsilon > 0 \) there exists \( M(\varepsilon)
 \in \mathbb{N}
 \), such that for every \( M \geq M(\varepsilon) \) there exists
 \( N(M,\varepsilon) \in \mathbb{N} \), such that for every \( N \geq
 N(M,\varepsilon)\)
 \begin{displaymath}
 \left\Vert{w}\right\Vert_N < \varepsilon,
 \end{displaymath}
 where \( w(n) \doteq \frac{1}{M} \sum_{m=1}^M \xi_1(a_1n + b_1 m) \xi_2(a_2 n + b_2 m)
 \ldots \xi_k(a_k n + b_k m) \) for every  \( n=1,2,\ldots,N\).
\end{prop}

\noindent Since the proof of Proposition \ref{main_prop} involves many
technical details, first we show how our main result follows from
it. Afterwards we state and prove all the lemmas necessary for the proof of
Proposition \ref{main_prop}.
\newline
We use an easy consequence of Proposition \ref{main_prop}.
\begin{corollary}
\label{corol_main_prop} Let \( A \) be a WM set. Let \( k \in
\mathbb{N} \), suppose 
\newline
\( (a_1,b_1),(a_2,b_2), \ldots, (a_k,b_k) \in
\mathbb{Z}^2 \) satisfy all requirements of
Proposition \ref{main_prop} and suppose \( f_1,\ldots,f_k \in
\mathbb{Z} \). Then for every \( \delta > 0 \) there exists \(
M(\delta)\) such that \( \forall \, M \geq M(\delta) \) there exists
\( N(M,\delta) \) such that \( \forall \, N \geq N(M,\delta)\) we
have
\[
\left| \Vert v \Vert_N - d^k(A) \right| < \delta,
\]
where \( v(n) \doteq \frac{1}{M} \sum_{m=1}^M 1_A(a_1n + b_1 m +f_1)
1_A(a_2 n + b_2 m +f_2)
 \ldots 1_A(a_k n + b_k m + f_k) \) for every  \( n=1,2,\ldots,N\).
\end{corollary}
\begin{proof}
We rewrite \( v(n) \) in the following form:
\[
v(n) =  \frac{1}{M} \sum_{m=1}^M (\xi_1(a_1n + b_1 m ) +d(A))
 \ldots  (\xi_k(a_k n + b_k m )+d(A)),
\]
for every \( n=1,2,\ldots,N\). We  introduce normalized WM
sequences \( \xi_i(n) = \xi(n+f_i) \) (of zero average), where \( \xi(n) = 1_A(n) - d(A) \). By use of triangular
inequality and Proposition \ref{main_prop} it follows that for big
enough \( M \) and \(N \) (which depends on \(M \)) \( \Vert v
\Vert_N \) is as close as we wish to \( d^k(A) \). This
finishes the proof.

 \hspace{12cm} \qed
\end{proof}

\noindent
\begin{proof} (of Theorem \ref{main_thm_lin}, \( \Lleftarrow \))
Let $A \subset \mathbb{N}$ be a WM set.
Without loss of generality, we can assume that for every \( r: 1 \leq r \leq l \) we have \( r \in F_r \).  
\newline
 It follows from Corollary \ref{corol_main_prop} that  the vector \( v \) defined by 
\[
v(n) \doteq \frac{1}{M} \sum_{m=1}^M 1_A(a_1n + b_1 m +f_1)
1_A(a_2 n + b_2 m +f_2)
 \ldots 1_A(a_l n + b_l m + f_l)
\]
for every  \( n=1,2,\ldots,N\), is not identically zero for big enough \(M \) and \( N \). But this is possible only if for some \(n,m \in \mathbb{N} \) we have
\[
(a_1n + b_1 m +f_1, a_2 n + b_2 m +f_2, \ldots, a_l n + b_l m + f_l) \in A^l.
\]
The latter implies that  $A^k$ intersects the affine subspace.

\hspace{12cm} \qed
\end{proof}

\noindent
\begin{proof} (of Theorem \ref{normal_theorem}, \( \Lleftarrow \))
For every \( r: \, 1 \leq r \leq l \) take all indices which comprise \( F_r \). 
Denote this sequence of indices by \( I_r \).
Denote \( c_r = \min_{i \in I_r} f_i \).  Let \( S_r \) be the set of all non-zero shifts of \( f_i, i \in F_r, \) centered at \( c_r \), i.e.,
\[
S_r = \{ f_i - c_r \, | \, i \in F_r, \, f_i  > c_r \}.
\]
For example, if the sequence of \( f_i \)'s where \( i \in F_1 \) is \( (-5,2,3,2,-5) \), then 
\( S_1 = \{7,8\} \).
\vspace{0.1in}

\noindent
Let \( A \) be a normal set. For every \(r \, : \, 1 \leq r \leq l \) we define sets \( A_r \) by 
\[
A_r = \{ n \in \N \cup \{0\} \, | \, n \in A \, \, and \,\,  n + s \in A, \, \forall s \in S_r\}.
\] 
Then \( A_r \) is no longer a normal set provided that \( S_r \neq \emptyset \) (\( d(A) = \frac{1}{2^{1+|S_r|}} \)). But, for all \( r \, : \, 1 \leq r \leq l \) the sets \( A_r \)'s are WM sets.

\noindent
Without loss of generality, assume that for every \( r \, : \, 1 \leq r \leq l \) we have \( r \in F_r \).

\noindent
From Proposition \ref{main_prop} it follows that for big enough \( M \) and \( N \) 
\[
\frac{1}{N} \sum_{n=1}^N \frac{1}{M} \sum_{m=1}^M 1_{A_1}(a_1n + b_1 m )
1_A(a_2 n + b_2 m)
 \ldots 1_A(a_l n + b_l m) \approx \prod_{r=1}^l d(A_r).
\]  
The latter ensures that there exist \( m,n \in \N \) such that 
\[
(a_1n +b_1m+f_1, \ldots, a_kn + b_km+f_k) \in A^k.
\]

\hspace{12cm} \qed
\end{proof}

\noindent Now we state and prove all the claims that are required in
order to prove Proposition \ref{main_prop}.



\begin{definition}
Let \( \xi \) be a WM-sequence (\( \xi \) is a generic point for a weakly mixing system \( (X_{\xi}, \mathbb{B}_{X_{\xi}}, \mu,T)\)) of zero average. The autocorrelation
function of \( \xi \) of length \( j \in \mathbb{N} \) with
the shifts \( \vec{i} = (i_1,i_2,\ldots,i_j) \in \Z^j \) and \( r \in \Z \) 
is the sequence \( \psi_{r,\vec{i}}^j \) which is defined by 
\[
\psi^j_{r,\vec{i}}(n)= \prod_{w \in \{0,1\}^j} \xi(n + r + w \cdot \vec{i}), \,\, n \in \N,
\]
where \( w \cdot \vec{i} \) is the usual scalar product in \( \mathbb{Q}^j \),
and 
\[
\psi^j_{r,\vec{i}}(n)= 0, \,\, n \leq 0.
\]
\end{definition}

\noindent
\begin{lemma}
\label{autocorrelation1} Let \( \xi \) be a WM-sequence of zero
average and suppose \( \varepsilon,\delta > 0, \, b \in \mathbb{Z} \setminus \{0\} \). Then for every
\( j \geq 1 \), \((c_1,c_2,\ldots,c_j) \in (\mathbb{Z} \setminus
\{0\})^j \) and  \( (r_1,r_2,\ldots,r_j) \in \mathbb{Z}^j \)
there exist \( I = I(\varepsilon,\delta,c_1,\ldots,c_n) \), a set \( S \subset [-I,I]^j \) of density at
least \( 1- \delta \) and \( N(S,\varepsilon) \in
\mathbb{N} \), such that for every \( N \geq N(S,\varepsilon)\)
there exists \( L(N,S,\varepsilon) \) such that for every \( L
\geq L(N,S,\varepsilon) \)
\[
  \frac{1}{L} \sum_{l=1}^L  \left( \frac{1}{N}  \sum_{n=1}^N
\psi^j_{r, (c_1i_1,\ldots, c_ji_j)}(l+ b
n)
 \right)^2 < \varepsilon
\]
for every \( (i_1,i_2,\ldots,i_j) \in S \), where \( r = \sum_{k=1}^j r_k \).
\end{lemma}
\begin{proof}
We note that it is sufficient to prove the lemma in the case \( c_1
= c_2 = \ldots = c_j = 1 \), since if the average of nonnegative
numbers over a whole lattice is small, then the average over a
sublattice of a fixed positive density is also small.
\newline
Recall that \( \xi \in X_{\xi} \doteq \overline{\{T^n \xi\}_{n=0}^{\infty}} \subset  supp({\xi})^{\mathbb{N}_0} \),
where \( T \) is the usual shift to the left on the dynamical system \( supp({\xi})^{\mathbb{N}_0} \), and
by the assumption that \( \xi \) is a WM-sequence of zero average it follows that \( \xi \) is a generic
point of the weakly mixing system \( (X_{\xi},\mathbb{B_{X_{\xi}}},\mu,T) \) and the function
\( f \, : \, f(\omega) \doteq \omega_0 \) has zero integral.
\newline
Denote \( \vec{i} = (i_1,\ldots,i_j) \).
\newline
We define functions \(
g_{r, \vec{i}},g^*_{r, \vec{i}}  \) on \( X_{\xi} \) by 
\[
g_{r, \vec{i}} \doteq \prod_{\epsilon \in V_j} T^{r+  \epsilon \cdot \vec{i}} \circ f,
\]
\[
g^*_{r, \vec{i}} = \prod_{\epsilon
\in V_j^{*}} T^{r + \epsilon \cdot \vec{i}} \circ f , \] 
where \( V_j \) is the \( j\)-dimensional discrete
cube \( \{0,1\}^j\) and \( V_j^{*} \) is the  \(
j\)-dimensional discrete cube except the zero point. 
\vspace{0.1in}

\noindent
Notice that 
\[ g_{r, \vec{i}} (T^n \xi) = \psi^j_{r,\vec{i}} (n). \]

\noindent
We use the following theorem which is a special case  of a multiparameter weakly mixing PET of Bergelson and McCutcheon (theorem A.1 in \cite{b_mctch}; it is also a corollary of Theorem 13.1 of 
 Host and Kra in \cite{host-kra}).
\newline \textit{ Let \( (X,\mu,T)\) be a weakly mixing system. Given an
integer \( k \) and \( 2^k \) bounded functions \( f_{\epsilon} \)
on \( X \), \(\epsilon \in V_k\) , the
functions \[
 \prod_{i=1}^k \frac{1}{N_i - M_i} \sum_{n \in [M_1,N_1)
\times \ldots [M_k,N_k)} \prod_{\epsilon \in V_k^{*}} T^{\epsilon_1 n_1 + \ldots \epsilon_k n_k} \circ 
f_{\epsilon}
\]
converge in \( L^2(\mu)\) to the constant limit  \[
 \prod_{\epsilon \in
V_k^*} \int_{X} f_{\epsilon} d\mu 
\]
when \( N_1 - M_1,\ldots,N_k-M_k\) tend to \( +\infty\). }

\noindent From this theorem  applied to the
weakly mixing system \( X_{\xi} \times X_{\xi} \) and the functions \(
f_{\epsilon}(x) = T^r \circ f \otimes T^r \circ f \) for every \( \epsilon \in V_j \),
 we obtain for every  Folner
sequence \( \{F_n\} \) in \( \mathbb{N}^j \) that an average over the multi-index \(\vec{i} = \{i_1,\ldots,i_j\}\)
 of \(
g^*_{r, \vec{i}} \otimes
g^*_{r, \vec{i}}\) on \( F_n \)'s
converges to zero in \( L^2(\mu) \) (the integral of \(T^r\circ f \otimes T^r \circ f \) is zero).
Thus 
\[
\int_{X_{\xi} \times X_{\xi}} \prod_{i=1}^j \frac{1}{N_i-M_i} \sum_{\vec{i} \in [M_1,N_1) \times \ldots \times [M_j,N_j)} g_{r,\vec{i}}(x) g_{r,\vec{i}}(y) d \mu (x) d \mu(y) = 
\]
\[
\prod_{i=1}^{j} \frac{1}{N_i-M_i} \sum_{\vec{i} \in [M_1,N_1) \times \ldots \times [M_j,N_j)} \left( 
\int_{X_{\xi}} g_{r,\vec{i}}(x) d \mu(x) \right)^2 \to 0, 
\] 
as \( N_1 - M_1, \ldots , N_j -M_j \to \infty \).

\noindent As a result we obtain the following statement:
\newline
 \textit{For every \( \varepsilon
> 0 \), \( j \in \mathbb{N} \) and every fixed \(
(r_1,r_2,\ldots,r_j) \in \mathbb{N}^j\), there exists a subset \(
R \subset \mathbb{N}^j \) of lower density equal to one, such
that
\begin{equation}
\label{eq:one_lemma311}
 \left( \int_{X_{\xi}}
g_{r,\vec{i}} d\mu
\right)^2 < \varepsilon
\end{equation}
 for every \( \vec{i} \in R \), where \(r = \sum_{k=1}^j r_j\).}
 
\noindent Recall that \textit{lower density} of a subset \( R \subset \mathbb{N}^j \) is defined to be
\[
d_{*}(R) = \liminf_{N_1-M_1,\dots,N_j-M_j \to \infty} \frac{\#\{R\cap[M_1,N_1)\times\ldots\times[M_j,N_j)\}}{\prod_{k=1}^j (N_k-M_k)}.
\]

\noindent Recall that \(
\psi^j_{r, \vec{i}}(l+ b n) =
g_{r, \vec{i}}\left( T^{l+bn} \xi
\right) \).
\newline
The definition of the sequences \( \psi^j \) implies
\[
\lim_{L \rightarrow \infty} \frac{1}{L} \sum_{l=1}^L \left(
\frac{1}{N} \sum_{n=1}^N
\psi^j_{r_1,\vec{i}}(l+bn) \right)^2
\]
\[
= \lim_{L \rightarrow \infty} \frac{1}{L} \sum_{l=1}^L \left(
\frac{1}{N} \sum_{n=1}^N \psi^j_{r_2, (\pm i_1, \ldots, \pm i_j)}(l \pm bn) \right)^2,
\]
for any \( r_1,r_2 \in \Z \), where \( \vec{i} = (i_1,\ldots,i_j)\).

\noindent
Therefore, in order to prove Lemma \ref{autocorrelation1}\, it
is sufficient to show the following:
\newline
\textit{For every \( \varepsilon,\delta > 0 \) and for any a priori
chosen \( b \in \mathbb{N} \)
 there exists  \(
I(\varepsilon,\delta) \in \mathbb{N} \), such that for every \( I
\geq I(\varepsilon,\delta) \) there exists a subset \( S \subset
[1,I]^j \) of density at least \( 1 - \delta \) \( \openrm
\)namely, we have \( \frac{|S \cap [1,I)^j|}{I^j} \geq 1-\delta
\)\( \closerm \) and  \( N(S,\varepsilon) \in
\mathbb{N} \), such that for every \( N \geq N(S,\varepsilon)\)
there exists \( L(N,S,\varepsilon) \in \mathbb{N} \) such that for
every \( L \geq L(N,S,\varepsilon) \) the following holds for
every \( \vec{i} \in S \):
\[
  \frac{1}{L} \sum_{l=1}^L  \left( \frac{1}{N}  \sum_{n=1}^N
\psi^j_{0, \vec{i}}(l+ b n)
 \right)^2 < \varepsilon.
\]
}

\noindent
Let \( b \in \mathbb{N} \). Continuity of
the function \newline \(
g_{0, \vec{i}} \) and genericity
of the point \( \xi \in X_{\xi} \) yield
\[
\lim_{L \rightarrow \infty} \frac{1}{L} \sum_{l=1}^L \left(
\frac{1}{N} \sum_{n=1}^N
\psi^j_{0, \vec{i}}(l+bn) \right)^2
\]
\[
= \lim_{L \rightarrow \infty} \frac{1}{L} \sum_{l=1}^L \left(
\frac{1}{N} \sum_{n=1}^N T^{bn}
g_{0, \vec{i}}\left( T^{l} \xi
\right)\right)^2
\]
\begin{equation}
\label{eq_lemm_2}
 = \int_{X_{\xi}} \left( \frac{1}{N} \sum_{n=1}^N T^{bn}
g_{0, \vec{i}} \right)^2 d \mu.
\end{equation}
By applying the von Neumann  ergodic theorem to the ergodic system
\newline
\( (X_{\xi},\mathbb{B},\mu,T^b) \) (ergodicity follows from weak-mixing of the original measure preserving system \( (X_{\xi},\mathbb{B},\mu,T) \)) we have
\begin{equation}
\label{imp_eq_22}
 \frac{1}{N} \sum_{n=1}^N T^{bn}
g_{0, \vec{i}} \rightarrow_{N
\rightarrow \infty}^{L^2(X_{\xi})} \int_{X_{\xi}}
g_{0, \vec{i}} d \mu.
\end{equation}
From (\ref{eq:one_lemma311}) there exists \( I(\varepsilon,\delta) \in \mathbb{N} \) big enough that for every \( I \geq I(\varepsilon,\delta) \) there
exists a set \( S \subset [1,I]^j \) of density at least \(
1-\delta \) such that
\[
 \left( \int_{X_{\xi}}
g_{0, \vec{i}} d \mu \right)^2 <
\frac{\varepsilon}{4}
\]
for all  \(\vec{i} \in S \).
\newline
From equation (\ref{imp_eq_22}) it follows that there exists \(
N(S,\varepsilon) \in \mathbb{N} \), such that for every \( N \geq
N(S,\varepsilon) \) we have
\[
\int_{X_{\xi}}\left( \frac{1}{N} \sum_{n=1}^N T^{bn}
g_{0, \vec{i}} \right)^2 d\mu <
\frac{\varepsilon}{2}
\]
for all \( \vec{i} \in S \).
\newline
Finally, equation (\ref{eq_lemm_2}) implies that there exists \(
L(N,S,\varepsilon) \in \mathbb{N} \), such that for every \( L \geq
L(N,S,\varepsilon) \) we have
\[
 \frac{1}{L} \sum_{l=1}^L \left(
\frac{1}{N} \sum_{n=1}^N
\psi^j_{0, \vec{i}}(l+bn) \right)^2
< \varepsilon
\]
for all \( \vec{i} \in S \).

\hspace{12cm} \qed
\end{proof}

\noindent The following lemma is a generalization of the previous
lemma to a product of several autocorrelation functions.
\begin{lemma}
\label{autocorrelation2}
 Let \(
\psi^{1,j}_{r_1,\vec{i}},
\ldots ,
\psi^{k,j}_{r_k,\vec{i}} \) be
autocorrelation functions of length \( j \) of WM-sequences \(
\xi_1, \ldots, \xi_k \) of zero average, 
\newline
\(
\{c_1^1,\ldots,c_j^1,\ldots, c_1^k,\ldots,c_j^k\} \in (\mathbb{Z}
\setminus \{0\})^{jk}\) and \( \varepsilon,\delta
> 0 \). Suppose\newline \( (a_1,b_1),(a_2,b_2), \ldots, (a_k,b_k) \in
\mathbb{Z}^2 \), such that \( a_i > 0 \) for all \( i: \, 1 \leq i \leq k \) and for every \( i \neq j \)
\begin{displaymath}
\det \left( \begin{array}{cc}
a_i & b_i \\
a_j & b_j  \\
\end{array}
\right) \neq 0.
\end{displaymath}
\textnormal{(}If \( k= 1 \) assume that \( b_1 \neq 0 \).\textnormal{)}
\newline
Then there exists \(I(\varepsilon,\delta) \in \mathbb{N} \), such
that for every \( I \geq I(\varepsilon,\delta) \) there exist \(
S \subset [-I,I]^j \) of density at least \( 1- \delta \), \(
M(S,\varepsilon) \in \mathbb{N} \), such that for every \( M \geq
M(S,\varepsilon) \) there exists \( X(M,S,\varepsilon) \in
\mathbb{N}\), such that for every \( X \geq X(M,S,\varepsilon) \)
\[
 \frac{1}{X} \sum_{x=1}^X \left( \frac{1}{M} \sum_{m=1}^M
\psi^{1,j}_{r_1,(c_1^1 i_1,\ldots, c_j^1 i_j)}(a_1x+ b_1 m) \ldots \psi^{k,j}_{r_k,(c_1^k i_1, \ldots, c_j^k i_j)}(a_kx+ b_k m)\right)^2 < \varepsilon
\]
for every \( (i_1,i_2,\ldots,i_j) \in S \).
\end{lemma}

\noindent \begin{proof}
The proof is by induction on \( k \). 
\vspace{0.1in}

\noindent
THE CASE \( k=1 \) (and
arbitrary \( j \)):
\vspace{0.1in}

\noindent
If \( a_1 = 1 \) then the statement of the lemma follows from Lemma \ref{autocorrelation1}. If \( a_1 > 1 \) then by Proposition \ref{wm_prop} of Appendix for a given \( \vec{i} = (i_1,\ldots,i_j) \in S \) we have 
\[
\lim_{X \to \infty} \frac{1}{X} \sum_{x=1}^X \left( \frac{1}{M} \sum_{m=1}^M \psi_{r_1,(c_1^1i_1,\ldots,c_j^1 i_j)}^{1,j}(a_1x+b_1m) \right)^2 = 
\]
\begin{equation}
\label{sub_eq_1}
\lim_{X \to \infty} \frac{1}{X} \sum_{x=1}^X \left( \frac{1}{M} \sum_{m=1}^M \psi_{r_1,(c_1^1i_1,\ldots,c_j^1 i_j)}^{1,j}(x+b_1m) \right)^2 
\end{equation}
(Limits exist by genericity of the point \( \xi \).)

\noindent By Lemma \ref{autocorrelation1} the right hand side of (\ref{sub_eq_1}) is small for large enough \( M \). So, for large enough \( X \) (depending on \( M \) and \( (i_1,\ldots, i_j)\)) the statement of the lemma is true. By finiteness of \( S \) we conclude that the statement of the lemma holds for \( k = 1 \).
\vspace{0.1in}

\noindent
GENERAL CASE (\( k > 1 \)):
\vspace{0.1in}

\noindent  
Suppose that the statement holds for \( k-1 \).
\newline
Denote 
 \[
 v_m(x) \doteq \psi^{1,j}_{r_1,(c_1^1 i_1, \ldots, c_j^1 i_j)}(a_1x+ b_1 m) \ldots
\psi^{k,j}_{r_k, (c_1^k i_1,\ldots, c_j^k i_j)}(a_kx+
b_k m). \] 
Let $\varepsilon, \delta > 0$. We show that there exists  \(
\mathbb{I}(\varepsilon,\delta) \in \mathbb{N} \) such that for every $\mathbb{I}  > \mathbb{I}(\varepsilon,\delta)$ a set $S \subset [-\mathbb{I},\mathbb{I}]^j$ of density at least $1-\delta$ can be chosen satisfying the following property:
\vspace{0.1in}

There exists $I(\varepsilon, S) \in \mathbb{N}$ such that for every $I > I(\varepsilon,S)$ there exists $M(I) \in \mathbb{N}$ such that for all $M > M(I)$ for a set of $i$'s in $\{1,2,\ldots,I\}$ of density at least $1 - \frac{\varepsilon}{3}$ we have
\begin{equation}
\label{vdr_crp}
\left | \frac{1}{M} \sum_{m=1}^M
<v_m,v_{m+i}>_X\right| < \frac{\varepsilon}{2}
\end{equation}
for all $(i_1,\ldots,i_j) \in S$.
\vspace{0.1in}

\noindent
The Van der Corput lemma (Lemma \ref{vdrCorput} of 
Appendix) finishes the proof.


\noindent 
Note that the set of ``good" \( i \)'s in the interval \(\{1,2,\ldots, I \}\) depends on \( (i_1,\ldots,i_j) \in S \). 
\vspace{0.1in}

\noindent
Denote
\[
\tilde{A} = \left | \frac{1}{M} \sum_{m=1}^M <v_m,v_{m+i}>_X\right|
\]
\[
 = \left| \frac{1}{X} \sum_{x=1}^X \frac{1}{M} \sum_{m=1}^M
\psi^{1,j+1}_{r_1, ( c_1^1 i_1, \ldots c_j^1 i_j, b_1
i )}(a_1x+ b_1 m) \ldots \psi^{k,j+1}_{r_k, (c_1^k i_1,\ldots, c_j^k i_j, b_k
i)}(a_kx+ b_k m) \right|.
\]
Denote \( y = a_1x+b_1m \). Assume that \( (a_1,b_1) = d \). Denote
\[
\tilde{B}_{y,m} = \psi^{1,j+1}_{r_1, ( c_1^1 i_1, \ldots c_j^1 i_j, b_1
i )}(y)
 \ldots
\psi^{k,j+1}_{r_k, (c_1^k i_1,\ldots, c_j^k i_j, b_k
i)}(a_k'y + b_k'm),
\]
where \(
a_p' = \frac{a_p}{a_1} \), \(b_p' = b_p - a_p'b_1 \), \( 2 \leq p \leq k \).
We rewrite \( \tilde{A} \) as follows:
\begin{equation}
\label{tilde_eq}
 \tilde{A} =
\left|a_1 \frac{1}{Y} \left(\sum_{l=0}^{\frac{a_1}{d}-1}
\sum_{y\equiv dl \mod a_1}^Y \frac{1}{M} \sum_{m \equiv \phi(l) \mod
\frac{a_1}{d}}^M \tilde{B}_{y,m}
\right)\right|
 + \delta_{X,M}.
\end{equation}
Here \( \phi \) is a bijection of 
\(\mathbb{Z}_{\frac{a_1}{d}}\) defined by the identity  
 \[
  \phi(l)\frac{b_1}{d} \equiv l  \, \left(\!\!\!\!\!\mod \frac{a_1}{d}\right), 
  \]
for every \( 0 \leq l \leq \frac{a_1}{d}-1 \), \( Y = a_1 X \),
\(a_p',b_p'\) as above and \( \delta_{X,M} \) accounts for the fact
that for small \( y\)'s and \( y \)'s close to \( Y \) there is a
difference between elements that are taken in the expression for \(
\tilde{A} \) and in the expression on the right hand side of equation (\ref{tilde_eq}).
 Nevertheless, we have \( \delta_{X,M} \rightarrow 0\)
if \( \frac{M}{X} \rightarrow 0 \).

\noindent
Denote 
\[
\tilde{C}_{y,m} = \psi^{2,j+1}_{r_2, (c_1^2i_1,\ldots,c_j^2i_j, b_2 i) }(a_2'y + b_2'm)
 \ldots
 \psi^{k,j+1}_{r_k, (c_1^k i_1,\ldots,c_j^k i_j,b_k
i )}(a_k'y + b_k'm).
\]
It will suffice to prove that there exists $\mathbb{I}(\varepsilon,\delta) \in \mathbb{N}$ such that for every $\mathbb{I} > \mathbb{I}(\varepsilon,\delta)$ we can find $S \subset [-\mathbb{I},\mathbb{I}]^j$ of density at least $1 - \delta$ with the following property:
\vspace{0.1in}

There exists $I(\varepsilon,S) \in \mathbb{N}$ such that for every $I > I(\varepsilon,S)$ there exists $M(I) \in \mathbb{N}$ such that for every $M > M(I)$ we can find $X(M) \in  \mathbb{N}$ such that for every $X > X(M)$ for a set of $i$'s in $\{1,2,\ldots,I\}$ of density at least $1 - \frac{\varepsilon}{3}$ we have  
\begin{equation}
\label{sub_eq_lemma}
 a_1\frac{1}{Y}  \sum_{y\equiv dl \mod a_1}^Y
\left(\frac{1}{M}
 \sum_{m \equiv \phi(l) \mod
\frac{a_1}{d}}^M \tilde{C}_{y,m}
 \right)^2 < \left( \frac{\varepsilon d}{3a_1} \right)^2
\end{equation}
for all \( 0 \leq l \leq \frac{a_1}{d}-1\), for all \( (i_1,\ldots,i_j) \in S \). 
\vspace{0.1in}


\noindent Note that it is enough to prove the latter statement for every particular \( l: \, 0 \leq l \leq \frac{a_1}{d} - 1 \).

\noindent
Denote the left hand side of inequality (\ref{sub_eq_lemma}) for a fixed \( l \) by \( \tilde{D}_l \).
\newline
Introduce new variables  \( z \) and \( n \), such that \( y = z a_1 + dl \) and \( m
= n \frac{a_1}{d}+ \phi(l) \). We obtain
\[
\tilde{D}_l = \frac{1}{Z} \sum_{z=1}^Z \left(\frac{d}{Na_1} \sum_{n=1}^N
\psi^{2,j+1}_{sh_2}\left(t_{n,z,l}^2\right)
 \ldots
\psi^{k,j+1}_{sh_k}\left(t_{n,z,l}^k\right)
  \right)^2
\]
\[
= \frac{1}{Z} \sum_{z=1}^Z \left(\frac{d}{Na_1} \sum_{n=1}^N
\psi^{2,j+1}_{sh_2}\left(a_2 z +c_2n + q_2\right)
 \ldots
\psi^{k,j+1}_{sh_k}\left(a_k z +c_kn + q_k\right)
  \right)^2,
\]
where \( sh_p = (r_p, (c_1^p i_1, \ldots, c_j^p i_j, b_p i)) \), 
\newline 
\( t_{n,z,l}^p = \frac{a_p (a_1z
+dl) + (a_1b_p-a_pb_1)(\frac{a_1}{d}n+\phi(l)) }{a_1}\), \( q_p =
\frac{a_p l d + (a_1 b_p-a_p b_1) \phi(l)}{a_1} \), \newline \( c_p =
\frac{a_1 b_p - a_p b_1}{d} \neq 0 \), \( Z = \frac{Y}{a_1} \) and
\( N = \frac{Md}{a_1} \). 
\newline
From the conditions on the function \( \phi \) it follows that \( q_p \in \Z , \, 2 \leq p \leq k\).

\noindent    From the conditions of the lemma we obtain for every \(
p \neq q , \, \, p,q > 1 \),
\[
\det \left( \begin{array}{cc}
a_p & c_p \\
a_q & c_q  \\
\end{array}
\right) = \frac{a_1 \det \left(\begin{array}{cc}
a_p & b_p \\
a_q & b_q  \\
\end{array}
\right)}{d} \neq 0.
\]
\newline
Therefore, \( \tilde{D}_l \) can be rewritten as
\[
\tilde{D}_l = \frac{1}{Z} \sum_{z=1}^Z \left(\frac{1}{Na_1} \sum_{n=1}^N
\phi_2\left(a_2 z +c_2n \right)
 \ldots
\phi_k\left(a_k z +c_kn \right)
 \right)^2,
\]
where $\phi_{\ell} = \psi^{\ell,j+1}_{r_{\ell} + q_{\ell}, (c_1^{\ell} i_1, \ldots, c_j^{\ell} i_j, b_{\ell} i)}, \,\, 2 \le \ell \le k$.
By the induction hypothesis the following is true.
\vspace{0.1in}

\noindent
\textit{There exists \(
\mathbb{I}_l(\varepsilon,\delta') \in \mathbb{N} \)  big enough, such
that for every \( \mathbb{I}_l \geq \mathbb{I}_l(\varepsilon,\delta')
\) there exist a subset \( S_l \subset
[-\mathbb{I}_l,\mathbb{I}_l]^{j+1} \) of density at least \( 1 -
\delta'^2 \) and \( N(S_l,\varepsilon) \in \mathbb{N} \),
such that for every \( N \geq N(S_l,\varepsilon) \) there
exists \(Z(N,S_l,\varepsilon) \in \mathbb{N} \), such that
for every \( Z \geq Z(N,S_l,\varepsilon)\) we have 
\begin{equation}
\label{ineq_last} \tilde{D}_l < \left( \frac{\varepsilon d}{3a_1}
\right)^2
\end{equation}
 for all \( (i_1,\ldots,i_j,i) \in S_l\).}
 \vspace{0.1in}

\noindent
For every \( (i_1,\ldots,i_j) \in [-\mathbb{I}_l,\mathbb{I}_l]^j \) we
denote by \( S_{i_1,\ldots,i_j}^l \) the fiber above \( (i_1,\ldots,i_j) \):
\[
S_{i_1,\ldots,i_j}^l = \{ i \in [-\mathbb{I}_l,\mathbb{I}_l] \,\, | \,\,
(i_1,\ldots,i_j,i) \in S_l \}.
\]
Then there exists a set \( T_l \subset [-\mathbb{I}_l,\mathbb{I}_l]^j \)
of density at least \( 1 - \delta' \), such that for every \(
(i_1,\ldots,i_j) \in T_l \) the density of \( S_{i_1,\ldots,i_j}^l \) is
at least \( 1 - \delta' \). Let \( \varepsilon, \delta > 0\). Take  \( \delta' < \min{(\frac{\varepsilon }{6}, \delta )}\)
 and \( \mathbb{I}
> \max{(I'(\varepsilon),\mathbb{I}_l(\varepsilon,\delta'))}
\) (\( I'(\varepsilon)\) is taken from the van der Corput lemma).
 
 \noindent
 Then it follows by (\ref{ineq_last}) that there exists
 \( M(T_l,\varepsilon,\delta) \in \mathbb{N}\), such that for every
 \( M \geq M(T_l,\varepsilon,\delta)\)
 there exists 
 \( X(M,T_l,\varepsilon,\delta) \in \mathbb{N} \), such that for
 every \( X \geq X(M,T_l,\varepsilon,\delta) \) the
 inequality (\ref{sub_eq_lemma}) holds
for every fixed \( (i_1,\ldots,i_j) \in T_l \) for a set of \( i
\)'s within the interval \(\{1,\ldots,\mathbb{I}\}\) of density at
least \( 1 - \frac{\varepsilon}{3} \). The lemma follows from the van der Corput lemma.

\hspace{12cm} \qed
\end{proof}

\noindent
\textbf{Proof of Proposition  \rm{\ref{main_prop}}.}
\newline
Denote \( v_m(n) \doteq \xi_1(a_1 n + b_1 m) \ldots \xi_k(a_k n +
b_k m) \). For every \( i \in \mathbb{N} \) we introduce \( \tilde{A} \) defined by  
\[
\tilde{A}  \doteq \left| \frac{1}{M} \sum_{m=1}^M <v_m,v_{m+i}>_N \right|.
\]
Then
\[
\tilde{A} = \left| \frac{1}{N} \sum_{n=1}^N  \frac{1}{M} \sum_{m=1}^M
\psi_{0,(b_1i)}^{1,1}(a_1 n+b_1 m) \ldots
\psi_{0,(b_ki)}^{k,1}(a_k n+b_k m)\right| ,
\]
where the functions \( \psi^{p,j} \)'s are autocorrelation functions of
the \( \xi_p \)'s of length \( j \).
\newline
By Lemma \ref{autocorrelation2} it follows that for every \( \varepsilon > 0 \) there exists \( I(\varepsilon) \in \N \) such that for every \( I \geq I(\varepsilon) \) there exist \( S \subset \{1,2,\ldots, I\}\) of density at least \( 1 - \frac{\varepsilon}{3} \) and \( M(S,\varepsilon)\) such that for every \( M \geq  M(S,\varepsilon)\) there exists \( N(M,S,\varepsilon) \) such that for every \( N \geq N(M,S,\varepsilon) \) we have 
\[
 \frac{1}{N} \sum_{n=1}^N  \left(\frac{1}{M} \sum_{m=1}^M
\psi_{0,(b_2i)}^{2,1}(a_2 n+b_2 m) \ldots
\psi_{0,(b_ki)}^{k,1}(a_k n+b_k m)\right)^2 \leq \varepsilon^2.
\]
The proposition follows from the van der Corput Lemma \ref{vdrCorput}.

\hspace{12cm} \qed
\section{Probabilistic constructions of WM sets}\label{lin_useful_constructions}
The goal of this section is to prove the necessity of the conditions
of Theorem \ref{main_thm_lin}. The following proposition is the main tool
for this task.
\begin{prop}
\label{3_constr_prop} Let \( a,b \in \mathbb{N} \), \( c \in \mathbb{Z} \) such that \( a
\neq b \). Then there exists a normal set \( A \) within which
the equation
\begin{equation}
\label{lin_eq_gen}
 ax=by + c
\end{equation}
is unsolvable, i.e., for every \( (x,y) \in A^2 \) we have \( ax
\neq by + c \).
\end{prop}
\begin{remark}
\textnormal{The proposition is a particular case of Theorem \ref{main_thm_lin}. It is a crucial ingredient in proving
the necessity direction of the theorem in general.}
\end{remark}
\begin{proof}
Let \( S \subset \mathbb{N} \). We construct from \( S \) a new
set \( A_S \) within which the equation \( ax = by + c \) is
unsolvable. 

\noindent 
Without loss of generality, suppose that \( a < b \).
\newline
Assume \( (a,b)=1 \) (the general case follows easily). It follows from \( (a,b) = 1 \) that (\ref{lin_eq_gen}) is solvable. Any solution \( (x,y) \) of
the
equation \( ax = by +c \) has restrictions on \( x \). Namely,  \( x \equiv
\phi(a,b,c) (\!\!\!\! \mod b) \), where \( \phi(a,b,c)\in \{ 0, 1,  \ldots, b-1\} \) is determined
uniquely. Let us denote \( l_0 \doteq  \phi(a,b,c)
\). We define inductively a sequence \( \{l_i\} \subset \mathbb{N} \cup \{0\}\).
If a pair \( (x,y) \) is a solution of
 equation (\ref{lin_eq_gen}) and \( y \in b^i\mathbb{N}+l_{i-1} \) then choose \( l_{i} \in \{0,1,\ldots,b^{i+1}-1\} \) such that
\( x \in b^{i+1} \mathbb{N} + l_i \).
\vspace{0.1in}

\noindent
Note that from \( (a,b)=1 \) it follows that \( (a,b^{i+1}) = 1 \). It is clear that if \( u,v \in \N \) satisfy \( (u,v) = 1 \) then for any \( w \in \Z \) there exists a solution \( (x,y) \in \N^2 \) of the equation \( ux = vy + w \). The latter implies that there exist \( x \in \N , y \in b^i \N + l_{i-1} \) such that \( ax = by + c \). Any such \( x \) should be a member of  \( b^{i+1}\N + l_i \). Note that \( l_i \) and \( l_{i-1} \) are connected by the identity 
\begin{equation}
\label{equiv_ls}
al_i \equiv b l_{i-1} + c \,\,(\!\!\!\! \mod b^{i+1} ).
\end{equation}
In addition, if \( x \in \N \) is given then the equation 
\[
ax \equiv by + c \,\,(\!\!\!\! \mod b^{i+1} )
\]
has at most one solution \( y \in \{0,1,\ldots, b^i-1\}\).
\vspace{0.1in}

\noindent
We define sets \( H_i \doteq b^i \mathbb{N} + l_{i-1} \, ; \, i \in \mathbb{N} \). We prove that for every
\( i \in \mathbb{N}, \, H_{i+1} \subset H_i \). All elements of \( H_{i+1} \) are in the same class
modulo \( b^{i+1} \), therefore all elements of \( H_{i+1} \) are in the same class modulo \( b^i \). So, if we show for some \( x \in H_{i+1} \) that \( x \equiv l_{i-1} ( \!\!\!\! \mod b^{i}) \)
then we are done. For \( i = 1 \) we know that if
\( y \in \mathbb{N}\) then any \( x \in \N \) such that \( (x,y) \) is a solution of the equation (\ref{lin_eq_gen})  has to be in \( H_1 \). Take \( x \in H_2\) such that
there exists \( y \in H_1 \) with  \( ax = by + c \). Then  \( x \in H_1 \). Therefore,
we have shown that \( H_2 \subset H_1 \). For \( i > 1 \) there exists \( x \in H_{i+1} \) such that
there exists \( y \in H_{i} \) with \( ax = by +c \). By induction \( H_{i} \subset H_{i-1} \). Therefore,
the latter \( y \) is in \( H_{i-1} \). Therefore, by construction of \( l_i \)'s we have that
\( x \in H_{i} \). This shows \( H_{i+1} \subset H_i \).
We define sets \( B_i; \, 0 \leq i < \infty \):
\[
  B_0= \mathbb{N} \setminus H_1,
  \]
  \[
  B_1=  H_1 \setminus  H_2
\]
\[
\ldots
\]
\[
  B_i =  H_i \setminus  H_{i+1}
\]
\[
  \ldots
\]
\newline
Clearly we have \( B_i \cap B_j = \emptyset \, \,, \forall
i \neq j \) and \( |\mathbb{N} \setminus (\cup_{i=0}^{\infty} B_i )|  = | \cap_{i=1}^{\infty} H_i | \leq 1 \).
 The latter is because
for every \( i \) the second element (in the increasing order) of \( H_i \) is
\( \geq b^i \).

\noindent
We define \( A_S = \bigcup_{i=0}^{\infty} A_i \), where
\( A_i \)'s are defined in the following manner:
 \[ A_0 \doteq S \cap B_0 , C_0 \doteq  B_0 \setminus A_0\]
\[ D_1 \doteq B_1 \setminus \{x \,|\, ax \in b B_0 + c\},
 A_1 \doteq  \left( B_1 \cap \{x\,|\, ax \in b C_0 + c\} \right) \cup \left( D_1 \cap S \right),\]
 \[
 C_1 \doteq  B_1 \setminus A_1\]
\[ \ldots \]
\[
D_i \doteq B_i \setminus \{x \,|\, ax \in b B_{i-1} + c\}, A_i = \left( B_i \cap
\{x\,|\, ax \in b C_{i-1} + c\} \right) \cup \left(D_i \cap S\right), \]
\[
C_i \doteq  B_i \setminus A_i
\]
\[
\ldots
\]
Here  it is worthwhile to remark that for every \( i, \,\, B_i = A_i \cup C_i\).
Therefore \( A_S \subset \cup_{i=0}^{\infty} B_i \).
\vspace{0.1in}

\noindent
If for some \( i \geq 1 \) we have \( y \in A_i \subset B_i = H_i \setminus H_{i+1} \), then any \( x\) with \( ax = by + c \) satisfies
\[
ax \equiv bl_{i-1}+c \,\,(\!\!\!\! \mod b^{i+1}).
\]
From \( (a,b^{i+1}) = 1 \) it follows that there exists a unique solution \( x \) modulo \( b^{i+1} \). By  identity 
(\ref{equiv_ls}) we have
\[
x \equiv l_i \,\,( \!\!\!\! \mod b^{i+1}).
\]
Thus \( x \in H_{i+1} \). 
\newline
If \( x \in H_{i+2} \), then 
\[
x \equiv l_{i+1} \,\,(\!\!\!\! \mod b^{i+2}).
\] 
Thus we have 
\[
al_{i+1} \equiv by+c \,\,(\!\!\!\! \mod b^{i+2}).
\]
By uniqueness of a solution ( \( y \) ) modulo \( b^{i+1} \) we get 
\[
y \equiv l_i \,\,( \!\!\!\! \mod b^{i+1}).
\]
Thus \( y \in H_{i+1}\). We have a contradiction, which shows that 
 \( x \in H_{i+1} \setminus H_{i+2} = B_{i+1} \).
\newline
The same argument works for \( y \in A_0 \subset B_0 \) and it shows that any \( x \) with 
\( ax = by +c \) satisfies \( x \in B_1\).
\vspace{0.1in}

\noindent
So, if \( y \in A_i \) (\( i \geq 0 \)) then any \( x \) with \( ax = by + c \) should satisfy \( x \in B_{i+1} \). By construction of \( A_S \),  \( x \not \in A_S \).
Thus  equation
(\ref{lin_eq_gen}) is not solvable in $A_S$. 

\noindent
We make the following claim:
\vspace{0.1in}

 \textit{For almost every subset \( S \) of \( \mathbb{N} \) the set
\( A_S \) is a normal set.}
\vspace{0.1in}

\noindent (The probability
measure on subsets of \( \mathbb{N} \) considered here is the product on \( \{0,1\}^{\infty} \) of
probability measures \( (\frac{1}{2},\frac{1}{2})\).)

\noindent
The tool for proving the claim is the following easy lemma (for a proof see Appendix, Lemma
 \ref{norm_lemma}).
\newline
\textit{A subset \( A \) of natural numbers is a normal set  if and only if  for any \(k \in (\mathbb{N} \cup \{0\}) \) and
any \( i_1 < i_2 <  \ldots < i_k \) we have
\begin{equation}
 \lim_{N \rightarrow \infty} \frac{1}{N} \sum_{n=1}^N \chi_A(n) \chi_A(n+i_1) \ldots \chi_A(n+i_k)
 =0,
 \label{eq:main_eq}
\end{equation}
where \( \chi_A(n) \doteq 2 \cdot 1_A(n) - 1\).}
\newline
First of all, we denote  \( T_N = \frac{1}{N} \sum_{n=1}^N
\chi_{A_S}(n) \chi_{A_S}(n+i_1) \ldots \chi_{A_S}(n+i_k) \).
Because of randomness of \( S \), \(T_N \) is a random variable. We will
prove that \( \sum_{N=1}^{\infty} E(T_{N^2}^2) < \infty \) and
this will imply by Lemma   \ref{tec_lemma}  that \( T_N \rightarrow_{N \rightarrow \infty} 0 \)
for almost every \( S \subset \mathbb{N} \).
\newline
\[
E(T_N^2) = \frac{1}{N^2} \sum_{n,m=1}^N E(\chi_{A_S}(n)
\chi_{A_S}(n+i_1) \ldots \chi_{A_S}(n+i_k) \chi_{A_S}(m)
 \ldots \chi_{A_S}(m+i_k)).
\]
Adding (removing) of a finite set to (from) a normal set does not affect the normality of the set. The set  \( \cup_i B_i \) might differ from \( \mathbb{N} \) by at most one element (\(|\cap_{i=1}^{\infty} H_i| \leq 1\)). This possible element  does not affect
the normality of \(A_S \) and we assume without loss of generality that \( \cap_{i=1}^{\infty} H_i = \emptyset\),
thus \( \mathbb{N} = \cup_{i=0}^{\infty} B_i \).
For every number \( n \in \mathbb{N} \) we define the chain of \( n \), \( Ch(n) \), to be the following finite sequence:
\newline
If \( n \in B_0 \), then \( Ch(n) = (n)\).
\newline
If \( n \in B_1 \), then two situations are possible. In the first one there exists a unique
\( y \in B_0 \) such that \( an = by +c \). We set \( Ch(n) = (n,y)=(n,Ch(y))\).
In the second situation we can not find such \( y \) from \( B_0\) and we set \(Ch(n) = (n) \).
\newline
If \( n \in B_{i+1}\), then again two situations are possible. In the first one there exists
\( y \in B_i \) such that \( an = by+c \). In this case we set
\(Ch(n)=(n,Ch(y))\). In the second situation there is no such \( y \) from \(B_i\).
In this case we set \( Ch(n) = (n) \). We define \( l(n)\)
to be the length of \( Ch(n) \).

\noindent For every \( n \in \mathbb{N} \)  we define the \textit{ancestor} of \( n \), \( a(n) \), to be the last
element of the chain of \( n \) (of  \(Ch(n)\)).
To determine  whether or not \( n \in A_S \) will
depend on whether \( a(n) \in S \). The exact relationship  depends on
the \( i \) for which \( n \in B_i \) and on the \( j \) for which \( a(n) \in B_j \) or in other words on the length
of \( Ch(n)\):
\( \chi_{A_S}(n) = (-1)^{i-j} \chi_S(a(n)) = (-1)^{l(n)-1} \chi_S(a(n))\).

\noindent We say that  \( n \) is a descendant of \( a(n) \).

 \noindent It is clear that \( E(\chi_{A_S}(n_1) \ldots
\chi_{A_S}(n_k)) \neq 0 \) (\( E(\chi_{A_S}(n_1) \ldots
\chi_{A_S}(n_k)) \in \{0,1\}\)) if and only if  every number \( a(n_i)
\) occurs an even number of times among numbers \( a(n_1),a(n_2),
\ldots, a(n_k)\).
\newline
We  bound the number of \( n,m\)'s inside the square \([1,N]
\times [1,N]\) such that \( E(\chi_{A_S}(n) \chi_{A_S}(n+i_1)
\ldots \chi_{A_S}(n+i_k) \chi_{A_S}(m) \chi_{A_S}(m+i_1) \ldots
\chi_{A_S}(m+i_k)) \neq 0 \).
\newline
For a given \( n \in [1,N] \) we  count all \( m \)'s inside
\( [1,N] \) such that for the ancestor of \( n \) there will be a chance
to have a twin among the ancestors of  \( n+i_1, \ldots,
n+i_k,m,m+i_1,\ldots,m+i_k \).
\newline
First of all it is obvious that in the interval \( [1,N]\) for a
given ancestor there can be at most \( \log_{\frac{b}{a}} N + C_1\) descendants, where \( C_1 \) is a
constant. For
all but a constant number of \( n \)'s it is impossible that among \( n+i_1,\ldots,n+i_k\)
there is the same ancestor as for \( n \). Therefore we should
focus on ancestors of the set \( \{ m, m+ i_1, \ldots, m+i_k\}\).
For a given \( n \) we might have at most \( (k+1) (\log_{\frac{b}{a}} N +C_1)\)
options for the number \( m \) to provide that  one of the elements
of \( \{ m, m+ i_1, \ldots, m+i_k\}\) has the same
ancestor as \( n \). Therefore for most of \( n \in [1,N] \)
(except maybe a bounded number \( C_2  \) of \( n \)'s which depends
only  on \( \{ i_1,\ldots,i_k \}\) and doesn't depend on \( N \))
we have at most \( (k+1) (\log_{\frac{b}{a}} N +C_1)\) possibilities for \( m \)'s
such that
\[
E(\chi_{A_S}(n) \chi_{A_S}(n+i_1) \ldots \chi_{A_S}(n+i_k)
\chi_{A_S}(m) \chi_{A_S}(m+i_1) \ldots \chi_{A_S}(m+i_k)) \neq 0.
\]
Thus we have
\[
E(T_N^2) \leq
\frac{1}{N^2} \left( \sum_{n=1}^N (k+1) (\log_{\frac{b}{a}} N +C_1)+ C_2 N \right) =
\frac{1}{N}((k+1)\log_{\frac{b}{a}} N + C_3),
\]
where \( C_3 \) is a constant.
This implies
\[
\sum_{N=1}^{\infty} E(T_{N^2}^2) < \infty.
\]
Therefore \( T_{N^2} \rightarrow_{N \rightarrow \infty} 0 \) for
almost every \( S \subset \mathbb{N} \). By Lemma \ref{tec_lemma} it
follows that \( T_N  \rightarrow_{N \rightarrow
\infty} 0 \) almost surely.
\newline
In the general case, where \( a,b \) are not relatively prime, if \( c \)  satisfies (\ref{lin_eq_gen}) then it should be
 divisible by \( (a,b) \). Therefore by dividing the equation (\ref{lin_eq_gen}) by \( (a,b) \)
 we reduce the problem to the previous case.

\hspace{12cm} \qed
\end{proof}

\noindent
We  use the following notation:
\newline
Let W be a subset of \( \mathbb{Q}^n\). Then for any increasing subsequence  \(I=(i_1,\ldots,i_p) \subset \{1,2,\ldots,n\}\) we define
\[
Proj_{I} W = W_{I}= \{ (w_{i_1},\ldots,w_{i_p}) \, | \, \exists w=(w_1,w_2,\ldots,w_n) \in W \}.
\]
We recall the notion of a cone. 
\begin{definition}
A subset \( W \subset \mathbb{Q}^n \) is called a \textbf{cone} if
\newline
\textnormal{(}a\textnormal{)} \( \forall w_1,w_2 \in W \) we have \( w_1+w_2 \in W \)
\newline
\textnormal{(}b\textnormal{)} \( \forall \alpha \in \mathbb{Q}: \, \alpha \geq 0 \) and \( \forall w \in W \) we have \( \alpha w \in W \). 
\end{definition}
The next step involves an  algebraic statement with a topological proof  which we have to
establish.
\begin{lemma}
\label{3_algebraic_lemma}
Let \( W \) be a non-trivial cone in \(
\mathbb{Q}^n \) which has the property that for every two vectors \(
\vec{a} = \{a_1,a_2,\ldots,a_n\}^t , \vec{b} =
\{b_1,b_2,\ldots,b_n\}^t \in W \)  there exist two coordinates \( 1
\leq i < j \leq n \) (depend on the choice of \(
\vec{a},\vec{b}\)) such that
\[
\det \left( \begin{array}{cc}
a_i & b_i \\
a_j & b_j  \\
\end{array}
\right) = 0.
\]
There exist  two coordinates \( i < j \) such that the
projection of \( W \) on these two coordinates is of dimension \(
\leq 1\) (\( dim_{\mathbb{Q}} \, Span \, Proj_{(i,j)} W \leq 1\)).
\end{lemma}
\begin{proof}
First of all \( W \) has positive volume  in 
\( V = Span W \) (Volume is Haar measure which normalized by assigning measure one to a unit cube and \( W \) contains a parallelepiped). Fix an arbitrary non-zero element \( \vec{x} \in W \).  For every \( i,j: \, 1 \leq i < j \leq n \) we define
the subspace 
\[ U_{i,j} = \{ \vec{v} \in V \, | \, Proj_{(i,j)} \vec{v} \in Span Proj_{(i,j)} \vec{x}\}.\]
\newline
From the assumptions of the lemma it follows   
that 
\[
W = \bigcup _{i,j;1 \leq  i < j \leq n} (W \cap U_{i,j}).
\]
 For every \(i \neq j \) we obviously have that the volume of \( U_{i,j} \) is either zero or \( U_{i,j} = V \). If we assume that the statement of the lemma does not hold then  \( U_{i,j} \neq V, \, \forall i \neq j \), and thus the volume of   \( U_{i,j},  \, \forall i \neq j \) is zero. We get a contradiction because a finite union of sets with zero volume cannot be equal to a set with positive volume.


\hspace{12cm} \qed
\end{proof}

\noindent
\begin{proof}\textit{(of Theorem \ref{main_thm_lin}, \(\Rrightarrow\))}
\newline
Assume that an affine subspace $\mathbb{A}$ of $\mathbb{Q}^k$ intersects $A^k$ for any  WM set $A \subset \mathbb{N}$.
\newline
First of all, we shift the affine space  to obtain a vector subspace, denote it by \( U \). The
linear space \( U \) must contain vectors with all positive
coordinates, since $\mathbb{A} \cap A^k$ must be infinite.

\noindent  Denote by \( W = \{
\vec{v} \in U \, | \, \langle\vec{v},\vec{e_i}\rangle \geq 0 \, , \,
\forall \, i: \, 1 \leq i \leq k \} \). \( W \) is a non-trivial cone. 
\vspace{0.1in}

\noindent
Assume that for every \( \vec{a}= (a_1,\ldots, a_k)^t, \vec{b} = (b_1,\ldots,b_k)^t \in W \) we have that \( \exists i,j: \,\, 1 \leq i < j \leq k \) such that 
\[
\det \left( \begin{array}{cc}
                    a_i & b_i \\
                    a_j & b_j  \\
            \end{array}
    \right) = 0.
\]
Then by Lemma \ref{3_algebraic_lemma} we deduce that 
there exist maximal subsets of coordinates \( F_1,\ldots,F_l \) (one of them, assume \( F_1 \), should have at least two coordinates) such that for every \( r \in
\{1,2\ldots,l\}\) we have  \(
V_{F_r} \doteq Span W_{F_r} \) is one dimensional.

 \noindent We fix \( r \, : \, 1 \leq r \leq l \). We  show that the projection on \( F_r \) of \( W + \vec{f} \) is on a  diagonal, where \( \vec{f} \in \Z^k \) is such that  \( U +\vec{f} = \mathbb{A}\). If the projection of \( W \) on
\( F_r \) is not on a diagonal then there exist two
coordinates \( i < j \) from \( F_r \) such that \( W_{(i,j)} =
\{(ax,bx) \, | \, x \in \mathbb{N}\}\) for some \( a \neq b \)
natural numbers. Therefore the projection of $\mathbb{A}$ on \((i,j)\) has the form \(\{(ax+f_1,bx+f_2) \, | \, x
\in \mathbb{N}\} \), where \( f_1,f_2 \) are integers. From
Proposition \ref{3_constr_prop} it follows that for any \( a,b,c \),
where \( a \neq b \), there exists a WM set \( A \) (even a normal set) such that the
equation \( ax = by + c \) is not solvable within \( A \). This
proves the existence of a WM set \( A_0 \) such that for every \( x
\in \mathbb{Z} \) we have \( (ax+f_1,bx+f_2) \not \in A_0^2 \) (introduce the new variables \( z_1,z_2 \) by \( (z_1,z_2)= (ax_1+f_1,bx+f_2) \) and take a normal set \( A_0 \) such that the equation \( az_2 = bz_1 + (af_2-bf_1) \) is unsolvable within \(A_0\)).

\noindent
Thus \( \forall i,j \in F_r : \,\, W_{(i,j)} = \{(ax,ax) \, | \, x \in \N\}\). 
\vspace{0.1in}

\noindent
 To prove that a shift is the same for all coordinates in \( F_r \) 
 we merely should
know that for any natural number \( c \) there exists a WM set \(
A_c \) such that inside \( A_c \) the equation \( x - y = c \) is
not solvable. The last statement is easy to verify.
\vspace{0.1in}

\noindent
Let \( j_r \in F_r \), \( \forall 1 \leq r \leq l \). Denote  \( I = (j_1,\ldots,j_l)\). We have proved that there exist \( g_1,\ldots,g_l \in \N \), \( c_1,\ldots,c_l \in \Z \) such that 
\[
(U + \vec{f})_I = \{ (g_1x_1 + c_1, \ldots, g_l x_l + c_l) \, | \, x_1,\ldots,x_l \in \mathbb{Q} \}. 
\]
It is clear that we can find \( \vec{a}, \vec{b} \) which satisfy all the requirements of Theorem \ref{main_thm_lin}.
This completes the proof.

\hspace{12cm} \qed

\end{proof}

\noindent
\begin{remark}
We have proved that if  an affine subspace $\mathbb{A} \subset \mathbb{Q}^k$ intersects $A^k$ for any normal set $A \subset \mathbb{N}$, then there exist \( \vec{a},\vec{b} \in \N^k \) and a partition \( F_1,\ldots,F_l \) of \( \{1,2,\ldots,k\}\) such that:
\newline
\textnormal{(a)} \( \forall r: \, 1 \leq r \leq l \) and \( \forall i \in F_r, \forall j \not \in F_r \) we have 
\[
\det \left( \begin{array}{cc}
                    a_i & b_i \\
                    a_j & b_j  \\
            \end{array}
    \right) \neq 0.
\] 
\textnormal{(b)} \( \exists \vec{f} \in \Z^k \) such that the set \(\{ n\vec{a} + m \vec{b} + \vec{f} \, | \, n,m \in \N \} \) is in $\mathbb{A}$.  

\noindent Thus, we have proved the direction ``\(\Rrightarrow\)" of Theorem \ref{normal_theorem}.
\end{remark}

\section{ Comparison with Rado's Theorem}\label{sub_sect_Rado}

We recall that the problem of solvability of a system of linear
equations in one cell of any finite partition of \( \mathbb{N} \) was solved
by Rado in  \cite{rado}. Such systems of linear equations
are called partition-regular.  We show that partition-regular
systems are  solvable within every WM set by use of Theorem \ref{main_thm_lin}. 
It is important to note that 
solvability of partition-regular linear systems of equations within WM sets can be 
shown directly (without use of Theorem \ref{main_thm_lin}) by use of the
technique of Furstenberg and Weiss
that was developed in their dynamical proof of Rado's theorem (see
\cite{furst3}). 
\newline
First of all we  describe Rado's regular systems. 
\begin{definition}
A rational \( p \times q \) matrix \( (a_{ij}) \) is said to be of
level \( l \) if the index set \( \{1,2,\ldots,q\}\) can be divided
into \( l \) disjoint subsets \( I_1,I_2,\ldots,I_l \) and rational
numbers \( c_j^r \) may be found for \( 1 \leq r \leq l \) and \( 1
\leq j \leq q \) such that the following relationships are
satisfied:
\[
\sum_{j \in I_1} a_{ij} = 0
\]
\[
\sum_{j \in I_2} a_{ij} = \sum_{j \in I_1} c_j^1 a_{ij}
\]
\[
\ldots
\]
\[
\sum_{j \in I_l} a_{ij} = \sum_{j \in I_1 \cup I_2 \cup \ldots \cup
I_{l-1}} c_j^{l-1} a_{ij}
\]
for \( i=1,2,\ldots,p\).
\end{definition}
\begin{theorem}\textit{(Rado)}
A system of linear equations is partition-regular if and only if  for some \(
l \) the matrix \( (a_{ij})\) is of level \( l \) and it is
homogeneous, i.e. a system of the form
\[
\sum_{j=1}^q a_{ij}x_j = 0, \hspace{0.5 in} i=1,2,\ldots,p.
\]
\end{theorem}
The following claim is the main result of this section.
\begin{prop}
\label{Rado_analog} A partition-regular system is solvable in every WM set.
\end{prop}
\begin{proof}
Let a system \( \sum_{j=1}^q a_{ij}x_j = 0, i=1,2,\ldots,p\) be
partition-regular. We will use the fact that the system is solvable for any
finite partition of \( \mathbb{N} \). First of all, the set of
solutions of a partition-regular system is a subspace of \( \mathbb{Q}^q \);
denote it by \( V \). It is obvious that \( V \) contains
vectors with all positive components. If for some \( 1 \leq i < j
\leq q \) we have \( Proj_{i,j}^{+} V \) (where \( Proj_{i,j}^{+} V
= \{(x,y)|x,y \geq 0 \hspace{0.1 in} \& \hspace{0.1 in} \exists
\vec{v} \in V: \, <\vec{v},\vec{e_i}>=x \, , \,
<\vec{v},\vec{e_j}>=y \} \)) is contained in a line, then \(
Proj_{i,j}^{+} V \) is diagonal, i.e. it is contained in \( \{(x,x)| x
\in \mathbb{Q}\} \). Otherwise, we can generate a partition of \(
\mathbb{N} \) into two disjoint sets \( S_1,S_2 \)  such that no \(
S_1^q \) and no \( S_2^q \) intersects \( V \):

\noindent  This partition is constructed by an iterative process.
Without loss of generality we may assume that the line is \( x = n y
\), where \( n \in \mathbb{N} \). The general case is treated in the
simillar way. We start with \( S_1 = S_2 = \emptyset \). Let \( 1 \in
S_1 \).
\newline
We ``color" the infinite geometric progression \( \{ n^m \, | \, m \in \mathbb{N} \} \) (adding elements
 to either \( S_1\) or
\( S_2 \)) in such way that there is no
\( (x,y) \) on the line from \( S_1^2,S_2^2 \).
Then we take a minimal element from \( \mathbb{N} \) which is still uncolored. Call it \( a \).
Add \( a \) to \( S_1 \). Next, ``color" \( \{ a n^m \, | \, m \in \mathbb{N} \} \).
\newline
Continuing in this fashion, we obtain the desired partition of $\mathbb{N}$. 

\noindent This contradicts the assumption that the given system is
partition-regular.
\newline
Let \( F_1,\ldots,F_l \) be a partition of \( \{1,2,\ldots,k\}\)  such that for every \( r \in \{1,\ldots,l\} \)
we have for every \( i \neq j \, , \, i,j \in F_r: \, \dim_{\mathbb{Q}} Span(Proj_{i,j}^{+} V) = 1 \), and for every \( r: \, 1 \leq r \leq l \), every \( i \in F_r \) .and for every \( j \not \in F_r \) we have \( \dim_{\mathbb{Q}} Span(Proj_{i,j}^{+} V) = 2\). 
 For every \( r: 1 \leq r \leq l \)  we choose
arbitrarily one representative index within \( F_r \) and denote it by \( j_r \) (\(j_r \in F_r \)).

\noindent 
Then there exist \( g_1,\ldots, g_l \in \N \) such that 
\[
V_I = \{ (g_1x_1,\ldots, g_l x_l) \, | \, x_1,\ldots,x_l \in \mathbb{Q} \}.
\]
The latter ensures that there exist  vectors \( \vec{a},\vec{b} \in V\) which satisfy all the requirements of Theorem
\ref{main_thm_lin} and, therefore, the system is solvable in every WM set.

\hspace{12cm} \qed
\end{proof}

\section{Appendix}\label{appendix}
\numberwithin{lemma}{section} \numberwithin{theorem}{section}
\numberwithin{prop}{section} \numberwithin{remark}{section}
\numberwithin{definition}{section} \numberwithin{corollary}{section}

\noindent In this section we prove all technical lemmas and
propositions that were used in the paper.
\newline
We start with the key lemma which is a finite modification of
Bergelson's lemma in \cite{berg_pet}. Its origin is in a lemma of
van der Corput.
\begin{lemma}
\label{vdrCorput}
 Suppose \(\varepsilon
>0 \) and \( \{u_{j}\}_{j=1}^{\infty } \) is a family
of vectors in  Hilbert space, such that \( \Vert u_j \Vert \leq 1 \,
\rm( 1 \leq j \leq \infty \rm)  \). Then there exists \(
I'(\varepsilon) \in \mathbb{N} \), such that for every \( I \geq
I'(\varepsilon) \) there exists \( J'(I,\varepsilon) \in \mathbb{N}
\), such that the following holds:
\newline
For \( J \geq J'(I,\varepsilon) \) for which we obtain
\[
\left| \frac{1}{J}\sum ^{J}_{j=1}\langle u_{j}, u_{j+i}\rangle\right| <
\frac{\varepsilon}{2}
 \]  
 for a set of \( i \)'s in the interval \(
\{1,\ldots,I\} \) of density \( 1 - \frac{\varepsilon}{3}\) we have
\[
\left\Vert \frac{1}{J}\sum _{j=1}^{J}u_{j}\right\Vert < \varepsilon.
\]
\end{lemma}
\begin{proof}
For an arbitrary \( J \) define \( u_{k}=0 \) for every \(
\textrm{k}<1 \) or \( k>J \). The following is an elementary
identity:
\[
\sum ^{I}_{i=1}\sum ^{J+I}_{j=1}u_{j-i}=I\sum
^{J}_{j=1}u_{j}.
\]
Therefore,  the inequality \( \left\Vert \sum
_{i=1}^{N}u_{i}\right\Vert ^{2}\leq N\sum _{i=1}^{N}\left\Vert
u_{i}\right\Vert ^{2} \) yields
 \[
\left\Vert I\sum ^{J}_{j=1}u_{j}\right\Vert ^{2}\leq (J+I)\sum
^{J+I}_{j=1}\left\Vert \sum ^{I}_{i=1}u_{j-i}\right\Vert ^{2}=
(J+I)\sum ^{J+I}_{j=1}\langle \sum ^{I}_{p=1}u_{j-p},\sum
^{I}_{s=1}u_{j-s}\rangle\]
 \[
=(J+I)\sum ^{J+I}_{j=1}\sum ^{I}_{p=1}\left\Vert u_{j-p}\right\Vert
^{2}+2(J+I)\sum ^{J+I}_{j=1}\sum
^{I}_{r,s=1;s<r}\langle u_{j-r},u_{j-s}\rangle=
 (J+I)(\Sigma _{1}+2\Sigma _{2}),
\]
where \( \Sigma _{1}=I\sum ^{J}_{j=1}\left\Vert u_{j}\right\Vert
^{2} \) by  the aforementioned elementary identity and
 \(
\Sigma _{2}=\sum ^{I-1}_{h=1}(I-h)\sum ^{J}_{j=1}\langle u_{j},u_{j+h} \rangle \).
The last expression is obtained by rewriting \( \Sigma_{2} \), where
\( h=r-s \). By dividing the foregoing inequality  by \( I^{2}J^{2}
\) we obtain
 \[
\left\Vert \frac{1}{J}\sum ^{J}_{j=1}u_{j}\right\Vert
^{2}<\frac{J+I}{IJ}+\frac{J+I}{J}\left( \frac{\varepsilon}{2} +
\frac{\varepsilon}{3} \right) = \frac{J+I}{J}\left( \frac{1}{I} +
\frac{5 \varepsilon}{6} \right).
\]
Choose \( I'(\varepsilon) \in \mathbb{N} \), such that
\(\frac{12}{\varepsilon}
 \leq I'(\varepsilon) \leq \frac{12}{\varepsilon}+1\). Then for
every \( I \geq I'(\varepsilon) \) we have \(  \frac{1}{I} + \frac{5
\varepsilon}{6} \leq \frac{11\varepsilon}{12}\). There exists \(
J'(I,\varepsilon) \in \mathbb{N} \), such that for every \( J \geq
J'(I,\varepsilon) \):  \( \frac{J+I}{J} < \frac{12}{11} \).
As a result, for every \( I \geq I'(\varepsilon) \) there exists \(
 J'(I,\varepsilon) \), such that for every \( J \geq
 J'(I,\varepsilon) \)
\[
\left\Vert \frac{1}{J}\sum ^{J}_{j=1}u_{j}\right\Vert ^{2} <
\varepsilon.
\]

\hspace{12cm} \qed
\end{proof}

\noindent The next proposition was used  in Section
\ref{lin_proof_suff}.
\begin{prop}
\label{wm_prop}
 Let \( A \subset \mathbb{N} \) be a WM-set. Then
for every integer \( a>0 \) and every integers \( b_1,b_2,\ldots,b_k
\) 
\begin{displaymath}
 \lim_{N \rightarrow \infty} \frac{1}{N}
\sum_{n=1}^N \xi(n+b_1)\xi(n+b_2) \ldots \xi(n+b_k) =
\end{displaymath}
\[
 \lim_{N \rightarrow \infty} \frac{1}{N} \sum_{n=1}^N
\xi(an+b_1)\xi(an+b_2) \ldots \xi(an+b_k),
\]
where \( \xi \doteq 1_A - \dd(A) \).
\end{prop}
\begin{proof}
Consider the weak-mixing measure preserving system \( ( X_{\xi},
\mathbb{B},\mu,T) \).
\newline
The left side of the equation in the proposition is \( \int_{X_{\xi}}
T^{b_1} f T^{b_2} f \ldots T^{b_k} f d\mu \), where \( f(\omega)
\doteq \omega_0 \) for every infinite sequence inside \( X_{\xi}\).
We make use of the notion of disjointness of measure preserving systems. By \cite{furst2} we know that
every weak-mixing system is disjoint from any Kronecker system which
is a compact monothethic  group with Borel \(\sigma\)-algebra, the
Haar probability measure, and the shift by a chosen
element of the group. In particular, every weak-mixing system is
disjoint from the measure preserving system \(
(\mathbb{Z}_a,\mathbb{B}_{\mathbb{Z}_a}, S , \nu)\), where
 \( \mathbb{Z}_a = \mathbb{Z}/ a\mathbb{Z} \), \( S(n)
\doteq n+1 (\mod a) \). The measure and the \( \sigma\)-algebra of the last
system are uniquely determined. Therefore, from  Furstenberg's
theorem
(see \cite{furst2}, Theorem I.6) it follows that the point \( (\xi,0) \in X_{\xi}
\times \mathbb{Z}_a \) is a generic point of the product system \(
(X_{\xi} \times \mathbb{Z}_a, \mathbb{B} \times
\mathbb{B}_{\mathbb{Z}_a}, T \times S, \mu \times \nu) \). Thus, for
every continuous function \( g \) on \( X_{\xi} \times \mathbb{Z}_a
\) we obtain
\[
\int_{X_{\xi} \times \mathbb{Z}_a} g(x,m) d\mu(x) d \nu(m) = \lim_{N
\rightarrow \infty}  \frac{1}{N} \sum_{n=1}^N g(T^n \xi, S^n 0).
\]
Let \( g (x,m) \doteq f(x)1_0(m) \), which is obviously continuous on
\( X_{\xi} \times \mathbb{Z}_a \). Then  genericity of the point \(
(\xi,0) \) yields
\[
\int_{X_{\xi} \times \mathbb{Z}_a} f(x)1_0(m) d\mu(x) d \nu(m) =
\frac{1}{a} \int_{X_{\xi}} f(x) d \mu(x) = \]
\[
\lim_{N \rightarrow \infty}  \frac{1}{N} \sum_{n=1}^N f(T^n \xi)
1_0(n) = \lim_{N \rightarrow \infty}  \frac{1}{a} \frac{1}{N}
\sum_{n=1}^N f(T^{an} \xi).
\]
Taking instead of the function \( f \) the continuous function \(
T^{b_1}f T^{b_2} f \ldots T^{b_k} f\)  in the definition of \( g \)
finishes the proof.

\hspace{12cm} \qed
\end{proof}



\noindent The next two lemmas are very useful for constructing
normal sets with specifical properties.
\begin{lemma}
\label{norm_lemma} Let \( A \subset \mathbb{N} \). Let \( \lambda(n)
= 2 1_A(n) - 1 \). Then \( A \) is a normal set \( \Leftrightarrow \)
for any \(k \in (\mathbb{N} \cup \{0\}) \) and any \( i_1 < i_2 <
\ldots < i_k \) we have
\[
 \lim_{N \rightarrow \infty} \frac{1}{N} \sum_{n=1}^N \lambda(n) \lambda(n+i_1) \ldots \lambda(n+i_k)
 =0.
\]
\end{lemma}
\begin{proof}
``\(\Rightarrow\)" If \( A \) is normal then any finite word \( w \in
\{ -1, 1\}^{*} \)
 has the ``right" frequency \( \frac{1}{2^{|w|}}\) inside
\( w_{A} \). This guarantees that ``half of the time" the function \(
\lambda(n) \lambda(n+i_1) \ldots \lambda(n+i_k) \) equals \(1 \) and
``half of the time" is equal to \( -1\).
 Therefore we get the desired conclusion.
\newline
``\( \Leftarrow \)"  Let \( w \) be an arbitrary finite word of plus
and minus ones: \( w = a_1 a_2 \ldots a_k \) and we have to prove
that \( w \) occurs in \( w_{A}\) with the frequency \( 2^{-k} \).
For every \( n \in \mathbb{N} \) the word \( w \) occurs in \( 1_A \) and starting from \( n \)
if and only if
\[
 \left\{ \begin{array}{lll} 1_A(n) = a_1 \\
 \ldots \\
1_A(n+k-1) = a_k
\end{array}
\right.
\]
The latter is equivalent to the following
\[
 \left\{ \begin{array}{lll} \lambda(n) = 2a_1 - 1 \\
 \ldots \\
\lambda(n+k-1) = 2 a_k -1
\end{array}
\right.
\]
The frequency of \( w \) within \( 1_A \) is equal to
\[
  \lim_{N \rightarrow \infty} \frac{1}{N} \sum_{n=1}^N \frac{\lambda(n) (2 a_1 -1) +1}{2} \ldots
  \frac{\lambda(n+k-1) (2 a_k - 1) + 1}{2}.
\]
The limit is equal to \( \frac{1}{2^k}\).

\hspace{12cm} \qed
\end{proof}

\begin{lemma}
\label{tec_lemma} Let \( \{ a_n \} \) be a bounded sequence. Let
 \( T_N = \frac{1}{N} \sum_{n=1}^N a_n \). Then \( T_N \)
converges to a limit \( t \) \( \Leftrightarrow \) there exists a
sequence of increasing indices \( \{ N_i \} \) such that \(
\frac{N_i}{N_{i+1}} \rightarrow 1 \) and \( T_{N_i} \rightarrow_{i
\rightarrow \infty} t \).
\end{lemma}

\textit{Current Address: \\
Department of Mathematics \\ University of Wisconsin-Madison
 \\ 480 Lincoln Dr. \\ Madison, WI 53706-1388 \\ USA \\
E-mail: afish@math.wisc.edu}


\begin{thebibliography}{150}

\bibitem{berg_pet} Bergelson, V. Weakly mixing PET. Ergodic Theory Dynam. Systems 7 (1987), no. 3, 337--349.
\bibitem{b_mctch} Bergelson, V.; McCutcheon, R. An ergodic IP polynomial Szemer\'{e}di theorem.  Mem. Amer. Math. Soc.  146  (2000),  no. 695.
\bibitem{fish1} Fish, A.  Random Liouville functions and normal
sets. Acta Arith. 120 (2005), no. 2, 191--196.
\bibitem{Fish_sumset_phenomenon} Fish, A. Polynomial largeness of sumsets and totally ergodic sets, see http://arxiv.org/abs/0711.3201.
\bibitem{fish4} Fish, A. Ph.D. thesis, Hebrew University, 2006.

\bibitem{furst2}  Furstenberg, H. Disjointness in ergodic theory, minimal sets, and a problem in Diophantine approximation.
 Math. Systems Theory 1 (1967), 1-49.
 \bibitem{furst4} Furstenberg, H. Ergodic behavior of diagonal measures and a theorem of Szemer\'{e}di on arithmetic
  progressions. J. d' Analys Math. 31 (1977), 204--256.
\bibitem{furst3} Furstenberg, H. Recurrence in Ergodic Theory and
Combinatorial Number Theory. Princeton Univ. Press 1981.
\bibitem{host-kra} Host, B.; Kra, B. Nonconventional ergodic averages
 and nilmanifolds. Ann. of Math. (2) 161 (2005), no. 1, 397--488.
\bibitem{rado} Rado, R. Note on combinatorial analysis. Proc. London Math. Soc. 48 (1943), 122--160.
\bibitem{schur} Schur, I. Uber die Kongruenz \( x^m + y^m \equiv z^m (mod p) \). Jahresbericht der Deutschen Math.-Ver.
25 (1916), 114--117.
\bibitem{szemeredi} Szemer\'{e}di, E.
On sets of
integers containing no $k$ elements in arithmetic progression.
 Collection of articles in memory of Juri\v{i} Vladimirovi\v{c} Linnik. Acta Arith. 27 (1975), 199--245.
\end{thebibliography}
\end{document}